\documentclass[10pt]{amsart}
\usepackage{amscd, amsmath, amsthm, amssymb}
\newcommand{\bbC}{\mathbb{C}}

\newcommand{\bbP}{\mathbb{P}}

\newcommand{\Aut}{\textup{Aut}}

\newcommand\Pic{{\text{Pic}}}

\newtheorem{theorem}{Theorem}[section]
\newtheorem{lemma}[theorem]{Lemma}
\newtheorem{proposition}[theorem]{Proposition}

\newtheorem{corollary}[theorem]{Corollary}
\theoremstyle{definition}     % italic or bold etc.

\theoremstyle{remark}
\newtheorem{remark}[theorem]{Remark}
\numberwithin{equation}{section}
\begin{document}
\title[Vanishing theorem on fake projective planes]
{A vanishing theorem on fake projective planes with enough automorphisms}

\author[J. Keum]{JongHae Keum }
\address{School of Mathematics, Korea Institute for Advanced
Study, Seoul 130-722, Korea } \email{jhkeum@kias.re.kr}
\thanks{Research supported by NRF}
\subjclass[2000]{14J29, 14F05}

\date{July 8, 2013}
\begin{abstract}  For every fake projective plane $X$ with automorphism group of order 21, we prove that
$H^i(X, 2L)=0$ for all $i$ and for every ample line bundle $L$ with $L^2=1$. 
For every fake projective plane with automorphism group of order 9, we prove  the same vanishing for every cubic root (and its twist by a 2-torsion) of the canonical bundle $K$. As an immediate consequence, there are exceptional sequences of length 3 on such fake projective planes.
\end{abstract}

\maketitle

%%%%%%%%%%%%%%%%%%%%%%%%%%%%%%%%%%%%%%%%%%%
%\setcounter{section}{0}, \section{Introduction}
A compact complex surface with the same Betti numbers as the
complex projective plane $\bf{P}_{\bbC}^2$ is  called {\it a fake projective plane} if it is not
isomorphic to  $\bf{P}_{\bbC}^2$.
The canonical bundle of a fake projective plane is ample. So a
fake projective plane is nothing but a surface of general type
with $p_g=0$ and $c_1^2=3c_2=9$. Furthermore, its universal cover
is the unit 2-ball in $\bbC^2$ by \cite{Aubin} and \cite{Yau} and its fundamental group is a co-compact arithmetic subgroup of PU$(2,1)$ by \cite{Kl}.

Recently, Prasad and Yeung \cite{PY} classified all possible
fundamental groups of fake projective planes. Their proof also shows that the automorphism group of a fake projective plane has order 1, 3, 9, 7, or 21. Then Cartwright and Steger (\cite{CS},  \cite{CS2}) have carried out group theoretic
enumeration based on computer to obtain more precise result: there are exactly 50 distinct fundamental groups, each corresponding to a pair of fake projective planes, complex conjugate to each other. They also have computed the automorphism groups of all fake projective planes $X$. In particular
$$\Aut(X)\cong \{1\},\,\, C_3,\,\, C_3^2\,\, {\rm or}\,\, 7:3,$$   where $C_n$ is the cyclic group of order $n$ and $7:3$ is the unique non-abelian group of
order $21$. Among the 50 pairs 34 admit a non-trivial group of automorphisms: 3 pairs have $\Aut\cong 7:3$, 3 pairs have $\Aut\cong C_3^2$ and 28 pairs have $\Aut\cong C_3$. For each pair of fake projective planes Cartwright and Steger \cite{CS2} have also computed the torsion group $H_1(X, \mathbb{Z})=$Tor$(H^2(X, \mathbb{Z}))=$Tor$(\Pic(X))$, which is the abelianization of the fundamental group. According to their computation a  fake projective plane with more than 3 automorphisms has no 3-torsion.

It can be shown (Lemma \ref{L0}) that if a fake projective plane $X$ has no 3-torsion in $H_1(X, \mathbb{Z})$, then the canonical class $K_X$ is divisible by 3 and has a unique cubic root, i.e., a unique line bundle $L_0$, up to isomorphism, such that $3L_0\cong K_X$. Its isomorphism class $[L_0]$ is fixed by $\Aut(X)$, since $\Aut(X)$ fixes the canonical class. 

For a fake projective plane $X$ an ample line bundle $L$ is called an {\it ample generator} if its isomorphism class $[L]$ generates  $\Pic(X)$ modulo torsion, or equivalently if the self-intersection number $L^2=1$. Any two ample generators differ by a torsion. We will use additive notation for tensoring line bundles, e.g.
$$L+M:=L\otimes M,\,\,\, mL:=L^{{\otimes}m},\,\,\,-L:=L^{-1}.$$

\begin{theorem}\label{main} Let $X$ be a fake projective plane with $\Aut(X)\cong 7:3$.  Then for every ample generator $L$ of $\Pic(X)$ and for any $i$ we have the vanishing
$$H^i(X, 2L)=0.$$ Moreover, for any $i$ and for the cubic root $L_0$ of $K_X$
$$H^i(X, 2L_0)=H^i(X, L_0)=0.$$ 
\end{theorem}

The proof uses the elliptic structure of the quotient of $X$ by the order 7 automorphism \cite{K08}. In the course of proof we also show that the $I_9$-fibre on the minimal resolution of the order 7 quotient has multiplicity 1, which was not determined in \cite{K08}.
This additional information will be useful when one tries to give a geometric construction of such elliptic surfaces and fake projective planes. So far no fake projective plane has ever been constructed geometrically. %that the $(3,3)$-elliptic surface case does not occur. 

\begin{theorem}\label{main2} Let $X$ be a fake projective plane with $\Aut(X)\cong C_3^2$. Let $L$ be an ample generator of $\Pic(X)$  such that $\tau^*(2L)\cong 2L$ for all $\tau\in 
Aut(X)$, $($e.g. the cubic root $L_0$ and its twist $L_0+t$ satisfy the condition for any 2-torsion bundle $t)$. Then for any $i$ we have the vanishing
$$H^i(X, 2L)=0.$$ Moreover, for any $i$ and for the cubic root $L_0$ of $K_X$
$$H^i(X, 2L_0)=H^i(X, L_0)=0.$$ 
\end{theorem}

The proof uses the structure of the quotient of $X$ by an order 3 automorphism \cite{K08}.

In both theorems, the core part is the vanishing $H^0(X, 2L)=0$. The key idea of proof is that if $H^0(X, 2L)\neq 0$, then $\dim H^0(X, 4L)\ge 4$. On the other hand,  the Rieman-Roch yields $\dim H^0(X, 4L)=3$.

\begin{corollary}\label{cor} Let $X$ be a fake projective plane with $\Aut(X)\cong 7:3$ or $C_3^2$. Let $L_0$ be the unique cubic root of $K_X$. Then the three line bundles
$${\mathcal O}_X,\,\,-L_0,\,\,-2L_0$$ form an exceptional sequence on $X$.
\end{corollary}

This is equivalent to that $H^i(X, 2L_0)=H^i(X, L_0)=0$ for all $i$,  hence follows from Theorem \ref{main} and \ref{main2}. (Since $L_0$ is a cubic root of $K_X$, the latter vanishings are equivalent to the single vanishing $H^0(X, 2L_0)=0$.) This confirms, for fake projective planes with enough automorphisms, the conjecture raised by Galkin, Katzarkov, Mellit and Shinder \cite{GKMS} that predicts the existence of an exceptional sequence of length 3 on every fake projective plane. Disjoint from our cases, N. Fakhruddin \cite{F} recently has confirmed the conjecture for the case of three 2-adically uniformised fake projective planes.

For an ample line bundle $M$ on a fake projective plane $X$, $M^2$ is a square integer. When $M^2\ge 9$, $H^0(X, M)\neq 0$ if and only if $M\ncong K_X$. This follows from the Riemann-Roch and the Kodaira vanishing theorem. When $M^2\leq 4$, $H^0(X, M)$ may not vanish, though no example of non-vanishing has been known. If it does not vanish, then it gives an effective curve of small degree.
The non-vanishing of $H^0(X, M)$ is equivalent to the  existence of certain automorphic form on the 2-ball.

\bigskip
{\bf Notation}
\begin{itemize}
\item $K_Y$: the canonical class of $Y$ \item $b_i(Y):$ the $i$-th
Betti number of $Y$ \item $e(Y):$ the topological Euler number of
$Y$  \item $q(X):=\dim H^1(X,\mathcal{O}_{X})$, the irregularity
of a surface $X$ \item $p_g(X):=\dim H^2(X,\mathcal{O}_{X})$, the
geometric genus of a surface $X$ \item curves of type $[n_1, n_2,
\ldots, n_l]$ : a string of smooth rational curves of self
intersection $-n_1, -n_2, \ldots, -n_l$
\item $\sim$:  numerical equivalence between divisors on a (singular) variety
\item $[D]$:  the linear  equivalence class of a divisor $D$
\item $[L]$ : the isomorphism class of a line bundle $L$
\end{itemize}

\section{Preliminaries}

\begin{lemma}\label{H0} Let $X$ be a fake projective plane. Assume that 
$H^0(X, 2L)=0$ for any ample generator $L\in \Pic(X)$. Then
$$H^1(X, 2L)=H^2(X, 2L)=0$$  for any ample generator $L$.
\end{lemma}

\begin{proof} The vanishing $H^0(X, 2L)=0$ implies $H^0(X, L)=0$.  Since $K-2L$ is an ample generator,
$$H^2(X, 2L)=H^0(X, K-2L)=0.$$
Finally by the Riemann-Roch, $$H^1(X, 2L)=0.$$
\end{proof}

\begin{lemma}\label{H02} Let $X$ be a fake projective plane such that $K_X$ is divisible by 3. Let $L_0$ be an ample generator such that $K_X\cong 3L_0$.   Assume that 
$H^0(X, 2L_0)=0$. Then for all $i$
$$H^i(X, 2L_0)=H^i(X, L_0)=0.$$  
\end{lemma}

\begin{proof} In this case  $K-2L_0=L_0$, thus $H^2(X, 2L_0)=H^0(X, L_0)=0.$
Then the vanishing of $H^1$ follows from the Riemann-Roch.
\end{proof}

\begin{lemma}\label{4L} Let $X$ be a fake projective plane.  Then
for any ample generator $L\in \Pic(X)$ and any torsion $t\in \Pic(X)$
$$h^0(X, 4L+t)=3.$$ 
\end{lemma}

\begin{proof} Since $4L+t-K_X$ is ample, Kodaira's vanishing theorem gives
$$H^1(X, 4L+t)=H^2(X, 4L+t)=0.$$
Thus the claim follows from the Riemann-Roch.
\end{proof}

\begin{lemma}\label{2L} Let $X$ be a fake projective plane.  Assume that $X$ has 2-torsions only. Then for any ample generator $L$ and for any automorphism  $\sigma\in \Aut(X)$
$$\sigma^*(2L)=2L.$$ 
\end{lemma}

\begin{proof} Since $L$ is ample, so is $\sigma^*(L)$. Since
$$(\sigma^*(L))^2=L^2=1,$$
$\sigma^*(L)$ is an ample generator. Thus $\sigma^*(L)-L$ is a torsion, hence
$$\sigma^*(L)= L+s$$ for some 2-torsion $s\in \Pic(X)$. Finally $2(L+s)=2L$. 
\end{proof}

\begin{lemma}\label{L0} Let $X$ be a fake projective plane. 
\begin{enumerate}
\item If $H_1(X, \mathbb{Z})$ does not contain a 3-torsion, then the canonical class $K_X$ is divisible by 3 and there is a unique cubic root of $K_X$.
\item If there is a line bundle  $L_0$ such that $K_X\cong 3L_0$, then for any automorphism  $\sigma\in \Aut(X)$
$$\sigma^*(L_0)=L_0\quad ({\rm modulo\,\,a}\,\,3{\rm -torsion}).$$ 
\end{enumerate}
\end{lemma}

\begin{proof} (1) Since $p_g(X)=q(X)=0$, the long exact sequence induced by the exponential sequence gives $\Pic(X)=H^2(X, \mathbb{Z})$. By the universal coefficient theorem, Tor$(H^2(X, \mathbb{Z}))= {\rm Tor}(H_1(X, \mathbb{Z}))$. So $X$ does not admit a 3-torsion line bundle. Let $L$ be an ample line bundle with $L^2=1$. Then $$K_X=3L+t$$ for some torsion line bundle $t$. Since the order of $t$ is coprime to 3, one can write $t=3t'$. This proves that $K_X$ is divisible by 3. The second conclusion follows from that two cubic roots of $K_X$ differ by a 3-torsion.

(2) Write
$\sigma^*(L_0)= L_0+s$ for some torsion $s\in \Pic(X)$. 
Since $\sigma^*$ preserves $K_X=3L_0$, we see that 
$3s=0.$ 
\end{proof}

\begin{remark}
(1) By a result of Koll\'ar (\cite{Ko}, p. 96) the 3-divisibility of  $K_X$ is  equivalent to the liftability of the fundamental group to SU$(2,1)$. Except 4 pairs of fake projective planes the fundamental groups lift to SU$(2,1)$ (\cite{PY} Section 10.4, \cite{CS},  \cite{CS2}). In the notation of \cite{CS}, these exceptional 4 pairs are the 3 pairs in the class $(\mathcal{C}_{18}, p=3, \{2\})$, whose automorphism groups are of order 3, and the one in the class $(\mathcal{C}_{18}, p=3, \{2I\})$, whose automorphism group is trivial.

(2) There are fake projective planes with a 3-torsion and with canonical class divisible by 3 \cite{CS2}.

\end{remark}

\subsection{Quotients of fake projective planes}

Let $X$ be a fake projective plane with a
non-trivial group $G$ acting on it. In \cite{K08}, all
possible structures of the quotient surface $X/G$ and its minimal
resolution were classified:

\bigskip
\begin{enumerate}
\item If $G=C_3$, then $X/G$ is a ${\mathbb Q}$-homology
projective plane with $3$ singular points of type
$\dfrac{1}{3}(1,2)$ and its minimal resolution is a minimal
surface of general type with $p_g=0$ and $K^2=3$. \item If
$G=C_3^2$, then $X/G$ is a ${\mathbb Q}$-homology projective plane
with $4$ singular points of type $\dfrac{1}{3}(1,2)$ and its
minimal resolution is a minimal surface of general type with
$p_g=0$ and $K^2=1$. \item If $G=C_7$, then $X/G$ is a ${\mathbb
Q}$-homology projective plane with $3$ singular points of type
$\dfrac{1}{7}(1,5)$ and its minimal resolution is a $(2,3)$-,
$(2,4)$-, or $(3,3)$-elliptic surface. \item If $G=7:3$, then
$X/G$ is a ${\mathbb Q}$-homology projective plane with $4$
singular points, where three of them are of type
$\dfrac{1}{3}(1,2)$ and one of them is of type
$\dfrac{1}{7}(1,5)$, and its minimal resolution is a $(2,3)$-,
$(2,4)$-, or $(3,3)$-elliptic surface.
\end{enumerate}

\bigskip
Here a ${\mathbb Q}$-homology projective plane is a normal
projective surface with the same Betti numbers as $\bbP_{\bbC}^2$ (cf.
\cite{HK1}, \cite{HK2}). A fake projective plane is a nonsingular
${\mathbb Q}$-homology projective plane, hence every quotient is
again a ${\mathbb Q}$-homology projective plane. An
$(a,b)$-elliptic surface is a relatively minimal elliptic surface
over $\bbP^1$ with $c_2=12$ having two multiple fibres of
multiplicity $a$ and $b$ respectively. It has Kodaira dimension 1
if and only if $a\ge 2, b\ge 2, a+b\ge 5$. It is an Enriques
surface iff $a=b=2$, and it is rational iff $a=1$ or $b=1$. All
$(a,b)$-elliptic surfaces have $p_g=q=0$, and by \cite{D} its
fundamental group is the cyclic group $C_d$, where $d$ is the
greatest common divisor of $a$ and $b$. A $(2,3)$-elliptic surface is called a Dolgachev surface.

\subsection{Normal surfaces with quotient singularities}

 Let $S$ be a normal projective surface with quotient singularities and $$f : S'
\rightarrow S$$ be a minimal resolution of $S$. The free part of the second cohomology group of $S'$,
$$H^2(S', \mathbb{Z})_{free} := H^2(S', \mathbb{Z})/(torsion)$$ becomes a unimodular lattice. For a quotient singular
point $p\in S$, let $$\mathcal R_p\subset H^2(S',
\mathbb{Z})_{free}$$ be the sublattice  spanned by the numerical
classes of the components of $f^{-1}(p)$. It is negative
definite, and its discriminant group
$${\rm disc}(\mathcal R_p):={\rm Hom}(\mathcal R_p, \mathbb{Z})/\mathcal R_p$$ is isomorphic to the
abelianization $G_p/[G_p, G_p]$ of the local fundamental group
$G_p$. In particular, the absolute value $|\det(\mathcal R_p)|$ of
the determinant
 of the intersection matrix of $\mathcal R_p$ is equal to the order
$|G_p/[G_p, G_p]|$. Define
 $$\mathcal R:= \oplus_{p \in Sing(S)} \mathcal R_p\subset H^2(S', \mathbb{Z})_{free}.$$ 

Quotient singularities are log-terminal, thus the adjunction formula can be written as $$K_{S'} \sim f^{*}K_S -
 \sum_{p \in Sing(S)}{\mathcal{D}_p},$$ where $\mathcal{D}_p = \sum(a_jA_j)$ is an effective
 $\mathbb{Q}$-divisor with $0 \leq a_j < 1$  supported on $f^{-1}(p)=\cup A_j$   for each singular point $p$.
It implies that
\[K^2_S = K^2_{S'} - \sum_{p}{\mathcal{D}_p^2}= K^2_{S'} +\sum_{p}{\mathcal{D}_pK_{S'}}.
\]
The coefficients $a_j$ of $\mathcal{D}_p$ can
be computed by solving the system of equations
$$\mathcal{D}_pA_j=-K_{S'}A_j=2+A_j^2$$  given by the adjunction formula for each exceptional curve
$A_j\subset f^{-1}(p)$.

Let $[n_1, n_2, ..., n_l]$ denote a Hirzebruch-Jung continued
fraction, i.e.,
 $$[n_1, n_2, ..., n_l]= n_1 - \dfrac{1}{n_2-\dfrac{1}{\ddots -
\dfrac{1}{n_l}}}.$$

For a fixed  Hirzebruch-Jung continued fraction $w = [n_1, n_2,
\ldots, n_l]$, we define
        \begin{enumerate}
        \item $|w|=q$, the order of
the cyclic singularity corresponding to $w$, i.e.,
$w=\dfrac{q}{q_1}$ with $1\le q_1<q, \,\, \gcd(q, q_1)=1$. Note
that $|w|$ is the absolute value of the determinant of the
intersection matrix corresponding to $w$. We also define
\item   $u_j :=  |[n_1, n_2, \ldots, n_{j-1}]|\,\,\,
                (2\leq j\leq l+1), \quad u_0=0, \,\,u_1=1$,
 \item   $v_j := |[n_{j+1}, n_{j+2}, \ldots, n_l]|
                \,\,\, (0\leq j\leq l-1), \quad v_l=1,
                \,\,v_{l+1}=0$.\\ Note that   $u_{l+1}=v_0=|[n_1, n_2, \ldots, n_{l}]|=|w|$.
 \end{enumerate}
\bigskip

For a cyclic singularity $p$, the coefficients of
$\mathcal{D}_p$ can be expressed in terms of $v_j$ and $u_j$.

\begin{lemma}\label{Dp}$($\cite{HK2}, Lemma 3.1$)$
Let $p\in S$ be a cyclic singular point of order $q$. Assume that
$f^{-1}(p)$ has $l$ components $A_1, \ldots, A_l$ with $A_i^2 =
-n_i$ forming a string of smooth rational curves
$\overset{-n_1}{\circ}-\overset{-n_2}\circ-\cdots-\overset{-n_l}\circ$.
Then \begin{enumerate} \item $\mathcal{D}_p
=\overset{l}{\underset{j = 1}{\sum}} \big( 1 - \dfrac{v_{j} +
      u_{j}}{q} \big) A_{j},$
      \item $\mathcal{D}_pK_{S'}=-\mathcal{D}_p^2
=\overset{l}{\underset{j = 1}{\sum}} \big( 1 - \dfrac{v_{j} +
      u_{j}}{q} \big) (n_{j}-2),$
\end{enumerate}
\end{lemma}

In the rest of this section, we assume that $S$ is a
$\mathbb{Q}$-homology projective plane with cyclic singularities. Then $p_g(S')=q(S')=0$, thus
$$Pic(S')\cong H^2(S', \mathbb{Z}).$$ Denote  by $$\bar{\mathcal
R}\subset Pic(S')_{free}:=Pic(S')/(torsion)$$ the primitive closure of $\mathcal
R\subset Pic(S')_{free}$.

\begin{lemma}\label{general}$($\cite{HK2}, Lemma 3.7$)$ Let $S$ be a $\mathbb{Q}$-homology
projective plane with cyclic singularities and $f:S' \rightarrow
S$ be a minimal resolution. Assume that $K_S$ is not numerically
 trivial. Then the following hold true.
\begin{enumerate}
\item $D:=|\det(\mathcal R)|
K_S^2$ is a nonzero square number.
\item ${\rm disc}(\bar{\mathcal R})$ is a cyclic group of order $|\det(\bar{\mathcal R})|=\frac{|\det(\mathcal R)|}{c^2}$ where $c$ is the order of $\bar{\mathcal R}/\mathcal R$.
\item Define $$D':=|\det(\bar{\mathcal R})|
K_S^2=\frac{D}{c^2}.$$ Then $Pic(S')_{free}$ is generated over
$\mathbb{Z}$ by the numerical equivalence classes of exceptional
curves, an element $T\in Pic(S')_{free}$ giving a generator of
$\bar{\mathcal R}/\mathcal R$ and a $\mathbb{Q}$-divisor of the
form
$$M = \frac{1}{\sqrt{D'}} f^*K_S +  z, $$ where $z$ is a generator of ${\rm
disc}(\bar{\mathcal R})$, hence of the form $z= \underset{p \in
Sing(S)}{\sum} b_p e_p$ for some integers $b_p$, where $e_p$ is a
generator of ${\rm disc}(\mathcal R_p)$.
\item For each singular point $p$, denote by
$A_{1, p}, A_{2, p}, \ldots, A_{l_p, p}$ the exceptional curves of
$f$ at $p$ and by $q_p$ the order of the local fundamental group
at $p$. Then every element $E\in Pic(S')_{free}$ can be written
uniquely as
\begin{equation}\label{E}
E \sim mM + \underset{p \in Sing(S)}{\sum} \overset{l_p}{\underset{i
= 1}{\sum}} a_{i, p} A_{i,p}
\end{equation}
 for some integer $m$
and some $a_{i,p}\in \frac{1}{c}\mathbb{Z}$ for all $i, p$.
\item $E$ is supported on $f^{-1}(Sing(S))$ if and only if $m=0$.
Moreover, if $E$ is effective (modulo a torsion) and not supported
on $f^{-1}(Sing(S))$, then $m>0$ when $K_S$ is ample, and $m<0$
when $-K_S$ is ample.
\end{enumerate}
\end{lemma}

Let $E$ be a divisor on $S'$. Then by Lemma \ref{general}(4), the
numerical equivalence class of $E$ can be written as the form
\eqref{E}. Then one can express the intersection numbers $EK_{S'}$
and $E^2$ in terms of the intersection numbers $EA_{j,p}$.

\begin{proposition}\label{formula}$($\cite{HK2}, Proposition
4.2$)$ Assume that  $K_S$ is not numerically trivial.
 Let $E$ be a divisor on $S'$. Write the numerical equivalence class of $E$ as the form \eqref{E}. Then the following hold
true.
    \begin{enumerate}
     \item $
          EK_{S'}
 = \dfrac{m}{\sqrt{D'}}K^2_S - \underset{p}{\sum} \overset{ l_p}{\underset{j = 1}{\sum}}
      \big( 1 - \dfrac{v_{j,p} + u_{j,p}}{q_p} \big) EA_{j,p}.$
        \item $ E^2 =\dfrac{m^2}{D'}K^2_S-  \underset{p}{\sum}
\overset{l_p}{\underset{j = 1}{\sum}}\Big(\overset{j}{\underset{k
= 1}{\sum}} \dfrac{v_{j,p}
  u_{k,p}}{q_p}(EA_{k,p}) + \overset{l_p}{\underset{k = j+1}{\sum}}
\dfrac{v_{k,p}  u_{j,p}}{q_p}(EA_{k,p})
 \Big)EA_{j,p}.$

 \noindent\medskip If, for each $p \in
                Sing(S)$, $E$ has a non-zero intersection number with at most $2$ components of $f^{-1}(p)$,
                i.e., $EA_{j,p}=0$ for $j\neq s_p, t_p$ for some $s_p$ and $t_p$ with $1\le s_p< t_p\le l_p$,
                then\\
                $ E^2  =   \dfrac{m^2}{D'} K^2_S - \underset{p}{\sum} \Big(  \dfrac{v_{s_p} u_{s_p}}{q_p}
                (EA_{s_p})^2+ \dfrac{v_{t_p} u_{t_p}}{q_p} (EA_{t_p})^2
                + \dfrac{2 v_{t_p} u_{s_p}}{q_p} (EA_{s_p})(EA_{t_p}) \Big).$
    \end{enumerate}
\end{proposition}

\begin{proposition}\label{3}$($\cite{K11}, Proposition
2.4$)$ Let $S$ be a ${\mathbb Q}$-homology projective plane with
$3$ singular points $p_1, p_2, p_3$ of type
$\dfrac{1}{7}(1,5)=[2,2,3]$. Let $A_{i1}$, $A_{i2}$, $A_{i3}$ be
the three components of $f^{-1}(p_i)$ of type $[2,2,3]$. Then the
following hold true.
\begin{enumerate} \item $K_S^2=\dfrac{9}{7},\,\,\, D=3^27^2,\,\,\,
D'=3^2$. \item If $S'$ contains a $(-2)$-curve $E$ not contracted
by $f:S'\to S$, then \begin{enumerate} \item
$\sum_{i=1}^3(A_{i1}E+2A_{i2}E)$ is divisible by $3$,  \item
$\sum_{i=1}^3(A_{i1}E+A_{i2}E+A_{i3}E)\ge 3$,
\end{enumerate}
\end{enumerate}
\end{proposition}

\section{Proof of Theorem \ref{main}}

The existence of a fake projective plane with $\Aut\cong 7:3$ was first proved in \cite{K06}, as a degree 21 cover of Ishida surface which is a $(2,3)$-elliptic surface.
 Mumford surface \cite{Mum} is also a degree 21 cover of Ishida surface, but not Galois. In terms of Prasad and Yeung \cite{PY}, Mumford surface  and Keum surface belong to the {\it same class}
in the sense that their fundamental groups are contained in the same maximal arithmetic subgroup of PU$(2,1)$. The ball quotient by this maximal subgroup has 4 singular   points, 3 of type $\frac{1}{3}(1,2)$ and one of type $\frac{1}{7}(1,5)$, and its minimal resolution is Ishida surface \cite{Ish}. Ishida surface is a Dolgachev surface, a simply connected surface of Kodaira dimension 1 with $p_g=q=0$.

Throughout this section, $X$ will denote a fake projective plane with $\Aut(X)\cong 7:3$. Let $\sigma_7$ and $\sigma_3$ be automorphisms of $X$, of order 7 and 3, respectively. Then
$$\Aut(X)=\langle \sigma_7,\,\,\, \sigma_3 \rangle.$$

According to the explicit computation of  Cartwright and Steger \cite{CS2}, there are 6 such surfaces (3 pairs from 3 different classes), and these pairs are distinguished by the torsion group:
$$H_1(X, \mathbb{Z})=C_2^3,\,\,C_{2}^4\,\,\,{\rm or}\,\,C_2^{6}.$$ 

In the following proof of Theorem \ref{main}, we will use the structure of the quotient of $X$ by an order 7 automorphism. The proof is split into the 3 cases: the minimal resolution of the quotient is a $(2,3)$-, $(2,4)$-, $(3,3)$-elliptic surface.

\subsection{The case of $(2,3)$-elliptic surface}

First we refine the result of \cite{K08} and \cite{K12} in the $(2,3)$-elliptic surface case.

\begin{theorem}\label{2,3}
Let $Z$ be a ${\mathbb Q}$-homology projective plane with $4$
singular points, $3$ of type $\frac{1}{3}(1,2)$ and one of type
$\frac{1}{7}(1,5)$. Assume that its minimal resolution $\tilde{Z}$
is a $(2,3)$-elliptic surface.
 Then the following hold true.
\begin{enumerate}
\item The triple cover $Y$ of $Z$ branched at the three singular
points of type $\frac{1}{3}(1,2)$ is a ${\mathbb Q}$-homology
projective plane with $3$ singular points of type
$\frac{1}{7}(1,5)$. The degree $7$ cover $X$ of $Y$ branched at the three singular
points is a fake projective plane.
\item The elliptic fibration on $\tilde{Z}$ has four $I_3$-fibres, whose union contains the eight exceptional $(-2)$-curves. 
\item The minimal resolution $\tilde{Y}$ of $Y$ is a
$(2,3)$-elliptic surface, where every fibre of the elliptic
fibration on $\tilde{Z}$ does not split. 
\item The elliptic fibration on $\tilde{Y}$ has four singular fibres, one of type $I_9$ and three of type $I_1$, and each fibre has the same multiplicity as the corresponding fibre on $\tilde{Z}$.
\item The $I_9$-fibre on $\tilde{Y}$ has multiplicity $1$.
\end{enumerate}
\end{theorem}

\begin{proof} The first 4 assertions are contained in \cite{K08}, Corollary 4.12, and \cite{K12}, Theorem 0.5. We need to prove the last assertion.

Let $y_1, y_2, y_3\in Y$ be the three singular points, and $$A_{i1},\,\,\, A_{i2},\,\,\,A_{i3}\subset \tilde{Y}$$ be the three exceptional curves of type
$[2,2,3]$ lying over $y_i$. Let $$A_{11}, A_{12}, B_1, A_{21}, A_{22}, B_2, A_{31}, A_{32},
B_3$$ be the nine components of the $I_9$-fibre $F_0$ in a
circular order. We need to prove that $F_0$ is non-multiple. Suppose that $F_0$ is multiple. Then its multiplicity is 2 or 3.

\medskip
(1) Suppose that the $I_9$-fibre $F_0$ on $\tilde{Y}$ has multiplicity 3.\\
Since a general fibre is numerically equivalent to $6K_{\tilde{Y}}$, we see that $$A_{11}+ A_{12}+ B_1+ A_{21}+ A_{22}+B_2+ A_{31}+A_{32}+B_3=F_0\sim 2K_{\tilde{Y}},$$ hence by  pushing forward to $Y$
$$B'_1+B'_2+B'_3\sim 2K_{Y},$$ where
$B'_i\subset Y$ is the image of $B_i$. The induced action of $\sigma_3$ on $\tilde{Y}$ has order 3 and rotates the  $I_9$-fibre. Let $$\pi: X\to X/\langle\sigma_7\rangle=Y$$ be the quotient map. Then 
$$\pi^*B'_1+\pi^*B'_2+\pi^*B'_3\sim 2\pi^*K_{Y}\equiv 2K_X.$$
The 3 curves $\pi^*B'_1$, $\pi^*B'_2$, $\pi^*B'_3$ are rotated by $\sigma_3$, but fixed by $\sigma_7$. 
Let $$L\in \Pic(X)$$ be a fixed ample generator. Then $2K_X\sim 6L$, hence
 $$\pi^*B'_1\equiv 2L+t$$
 for some 2-torsion $t$. Here we use that $X$ has 2-torsions only. By Lemma \ref{2L}
  $$\pi^*B'_1=\sigma_7^*(\pi^*B'_1)\equiv \sigma_7^*(2L)+\sigma_7^*(t)=2L+\sigma_7^*(t),$$
hence $$t=\sigma_7^*(t).$$  We know that
$$\pi_1(Y)\cong\pi_1(\tilde{Y})=\{1\},$$ in particular, $X$ cannot have a $\sigma_7^*$-invariant torsion. Thus $t=0$ and
$$\pi^*B'_1\equiv \pi^*B'_2\equiv \pi^*B'_3\equiv  2L.$$
Next we need to determine the intersection numbers $$k_{ij}:=A_{i3}B_j$$
for $i, j =1, 2, 3$. Since  $A_{i3}$ is a $(-3)$-curve, it is a
6-section, so 
$$2=A_{i3}F_0=k_{i1}+k_{i2}+k_{i3}+1,$$
so $k_{i1}+k_{i2}+k_{i3}=1$ for each $i$. Applying Proposition \ref{formula}(1)
to $E=B_j$, we get
$$0=B_jK_{\tilde{Y}}=\dfrac{m_j}{3}\dfrac{9}{7}-\dfrac{3}{7}(k_{1j}+k_{2j}+k_{3j}+1),$$
so $$m_j=k_{1j}+k_{2j}+k_{3j}+1,$$ where $m_j$ is the leading
coefficient of $B_j$ in the form \eqref{E}. Now by Proposition \ref{formula}(2),
$$-2=B_j^2=\dfrac{m_j^2}{9}\dfrac{9}{7}-\dfrac{1}{7}\{3(k_{1j}^2+k_{2j}^2+k_{3j}^2)+11+4k_{jj}+2k_{j+1,j}\},$$
so
$$m_j^2+14=3(k_{1j}^2+k_{2j}^2+k_{3j}^2)+11+4k_{jj}+2k_{j+1,j},$$
where $k_{43}=k_{13}$. It is easy to check that this system has a
unique solution;
$$(k_{ij})=
\begin{pmatrix} 1 & 0&0\\0&1&0 \\ 0&0 & 1 \end{pmatrix}.$$
In particular, $y_j,\, y_{j+1} \in B_j'$ and $y_{j+2}\notin B_j'$   
for $j=1, 2, 3$, where $y_{k+3}=y_k$. Let $x_j=\pi^{-1}(y_j)\in X$ be the fixed points of $\sigma_7$. Then  for $j=1, 2, 3$
$$x_j,\, x_{j+1} \in \pi^*B_j',\qquad x_{j+2}\notin\pi^*B_j',$$ 
where $x_{k+3}=x_k$.
Now let $g_j\in H^0(X, 2L)$ be a section giving the divisor $\pi^*B'_j$. Then the 4 sections
$$g_1^2,\, g_2^2,\, g_1g_2,\, g_3^2\in H^0(X, 4L)$$  
are linearly independent, which can be easily checked by evaluating any given linear dependence at each $x_j$. This implies $h^0(X, 4L)\ge 4$, contradicting Lemma \ref{4L}. 

\medskip
(2) Suppose that the $I_9$-fibre $F_0$ on $\tilde{Y}$ has multiplicity 2.\\ 
In this case, 
$$A_{11}+ A_{12}+ B_1+ A_{21}+ A_{22}+B_2+ A_{31}+A_{32}+B_3=F_0\sim 3K_{\tilde{Y}},$$
hence
$$B'_1+B'_2+B'_3\sim 3K_{Y},$$
and
$$\pi^*B'_1\sim\pi^*B'_2\sim\pi^*B'_3\sim 3L.$$
Write $$\pi^*B'_1\equiv 3L_0+t$$
where $L_0$ is an ample  generator such that $K_X=3L_0$ and $t$ is a torsion.
We know that $\pi^*B'_1$ is $\sigma_7^*$-invariant.   Since $L_0$ is $\sigma_7^*$-invariant by Lemma \ref{L0}, so is $t$.
Since $Y=X/\langle\sigma_7\rangle$ is simply connected, $X$ cannot have a $\sigma_7^*$-invariant torsion. Thus $t=0$. Then  
 $$\pi^*B'_1\equiv 3L_0=K_X,$$ hence $K_X$ is effective,
 contradicting $p_g(X)=0$. This completes the proof of Theorem \ref{2,3}.
\end{proof}

\bigskip
{\bf Proof of Theorem \ref{main}.}

\medskip
By Lemma \ref{H0}, it is enough to show that $H^0(X, 2L)=0$. Suppose that $$H^0(X, 2L)\neq 0.$$ By Lemma \ref{2L}, $\sigma_7^*(2L)\cong 2L$, i.e. 
$\sigma_7^*$ acts on the projective space ${\bf P}H^0(X, 2L)$. Every finite order automorphism of a projective space has a fixed point, so there is a $\sigma_7^*$-invariant curve $$C\in |2L|,$$ possibly reducible. Then there is a curve $$C'\subset Y=X/\langle\sigma_7\rangle$$ such that $C=\pi^*C'$. Since $C'^2=4/7$, we see that
$$C'\sim \dfrac{2}{3}K_Y.$$
Let $\tilde{C}\subset \tilde{Y}$ be the proper transform of $C'$ under the resolution $$f: \tilde{Y}\to Y.$$ Then
$$\tilde{C}\sim f^*C'-\sum_{k=1}^3{\mathcal{C}_k}\sim \dfrac{2}{3}f^*K_Y-\sum_{k=1}^3{\mathcal{C}_k}$$
where $\mathcal{C}_k$ is a $\mathbb{ Q}$-divisor supported on  $A_{k1}\cup A_{k2}\cup  A_{k3}$. For any exceptional curve $A_{ij}$, we have $A_{ij}f^*C'=0$, hence
$$0\le A_{ij}\tilde{C}=-A_{ij}\sum_{k=1}^3\mathcal{C}_k=-A_{ij}\mathcal{C}_i.$$
We know that 
$$K_{\tilde{Y}}\sim f^*K_Y-\frac17(A_{11}+2A_{12}+3A_{13})-\frac17(A_{21}+2A_{22}+3A_{23})- \frac17(A_{31}+2A_{32}+3A_{33}).$$
It follows that 
$$K_{\tilde{Y}}\tilde{C}=(f^*K_Y)(f^*C')+\sum_{i=1}^3\frac17(A_{i1}+2A_{i2}+3A_{i3})\mathcal{C}_i\le (f^*K_Y)(f^*C')=K_YC'=\dfrac{6}{7}.$$
Since $\tilde{Y}$ is a relatively minimal elliptic surface with Kodaira dimension 1, $K_{\tilde{Y}}$ is numerically effective, thus $K_{\tilde{Y}}\tilde{C}$ is a non-negative integer, hence
$$K_{\tilde{Y}}\tilde{C}=0.$$
Thus the curve $\tilde{C}$ is contained in a union of fibres. 

On the other hand, since the $I_9$-fibre has multiplicity 1, 
$$B'_1+B'_2+B'_3\sim 6K_{Y}$$
and
$$\pi^*B'_1+\pi^*B'_2+\pi^*B'_3\sim 6\pi^*K_{Y}\equiv 6K_X\sim 18L,$$
hence
$$\pi^*B'_1\sim\pi^*B'_2\sim\pi^*B'_3\sim 6L.$$
If $F_1$ is the reduced curve of the fibre with multiplicity  2 or 3, then it is irreducible and the image
$F_1'\subset Y$ satisfies
$F_1'\sim 3K_{Y}$ or $2K_{Y}$ and hence
$$\pi^*F_1'\sim 9L\quad {\rm or}\quad 6L.$$
Thus no union of irreducible components of fibres can be equal to 
the curve $\tilde{C}$.  The proof is completed.

\begin{remark}\label{rem23} One can determine the intersection numbers $$k_{ij}:=A_{i3}B_j,$$
where $A_{11}, A_{12}, B_1, A_{21}, A_{22}, B_2, A_{31}, A_{32},
B_3$ are the nine components of the $I_9$-fibre $F_0$ as above.
Since the $I_9$-fibre has multiplicity $1$, we
get
$$k_{i1}+k_{i2}+k_{i3}=5,$$ $$m_j=k_{1j}+k_{2j}+k_{3j}+1,$$
$$m_j^2+14=3(k_{1j}^2+k_{2j}^2+k_{3j}^2)+11+4k_{jj}+2k_{j+1,j}.$$
There are many solutions to this system. Among them two are symmetric with respect to the order 3 rotation $(A_{13},
B_1)\to (A_{23}, B_2)\to (A_{33}, B_3)\to (A_{13}, B_1)$;
$$(k_{ij})=
\begin{pmatrix} 2&1&2\\2&2&1 \\ 1&2&2 \end{pmatrix},\,\,\,\begin{pmatrix}1&3&1\\1&1&3 \\3&1&1 \end{pmatrix}.$$
Each of the two cases leads to a fake projective plane as shown in \cite{K06}. 
\end{remark}

\subsection{The case of $(2,4)$-elliptic surface}

First we refine the result of \cite{K08} and \cite{K11} in the $(2,4)$-elliptic surface case.

\begin{theorem}\label{2,4}
Let $Z$ be a ${\mathbb Q}$-homology projective plane with $4$
singular points, $3$ of type $\frac{1}{3}(1,2)$ and one of type
$\frac{1}{7}(1,5)$. Assume that its minimal resolution $\tilde{Z}$
is a $(2,4)$-elliptic surface.
 Then the following hold true.
\begin{enumerate}
\item The triple cover $Y$ of $Z$ branched at the three singular
points of type $\frac{1}{3}(1,2)$ is a ${\mathbb Q}$-homology
projective plane with $3$ singular points of type
$\frac{1}{7}(1,5)$. The degree $7$ cover $X$ of $Y$ branched at the three singular
points is a fake projective plane.
\item The elliptic fibration on $\tilde{Z}$ has four $I_3$-fibres, whose union contains the eight exceptional $(-2)$-curves. 
\item The minimal resolution $\tilde{Y}$ of $Y$ is a
$(2,4)$-elliptic surface, where every fibre of the elliptic
fibration on $\tilde{Z}$ does not split. 
\item The elliptic fibration on $\tilde{Y}$ has four singular fibres, one of type $I_9$ and three of type $I_1$, and each fibre has the same multiplicity as the corresponding fibre on $\tilde{Z}$.
\item The $I_9$-fibre on $\tilde{Y}$ has multiplicity $1$.
\end{enumerate}
\end{theorem}

\begin{proof} The first 4 assertions are contained in \cite{K08}, Corollary 4.12, and \cite{K11}, Theorem 1.1 and 4.1. We need to prove the assertion (5). 

For a $(2,4)$-elliptic surface the general fibre $F$ is numerically equivalent to $4K$. Thus any $(-3)$-curve is a 4-section, since it has intersection number 1 with the canonical class $K$. 

 If the $I_9$-fibre on $\tilde{Y}$ has multiplicity 4. The $(-3)$-curve $A_{i3}$, being a 4-section, intersects only one component of the  $I_9$-fibre, namely $A_{i2}$, thus the curve $B_1$ does not meet $A_{i3}$ for any $i$ and hence
$$\sum_{i=1}^3(A_{i1}B_1+A_{i2}B_1+A_{i3}B_1)=2,$$ 
contradicting Proposition \ref{3}.

Suppose that the $I_9$-fibre on $\tilde{Y}$ has multiplicity 2. This case can be ruled out by a similar argument as in the $(2, 3)$-elliptic surface case.
 Since a general fibre is numerically equivalent to $4K_{\tilde{Y}}$, we see that $$A_{11}+ A_{12}+ B_1+ A_{21}+ A_{22}+B_2+ A_{31}+A_{32}+B_3=F_0\sim 2K_{\tilde{Y}},$$ hence 
$$\pi^*B'_1+\pi^*B'_2+\pi^*B'_3\sim 2\pi^*K_{Y}\equiv 2K_X.$$
The 3 curves $\pi^*B'_1$, $\pi^*B'_2$, $\pi^*B'_3$ are rotated by the order 3 automorphism $\sigma_3$, but fixed by the order 7 automorphism $\sigma_7$. 
Let $L$ be a fixed ample generator of $\Pic(X)$. Then 
 $$\pi^*B'_1\equiv 2L+t$$
 for some  $\sigma_7^*$-invariant torsion $t$. Since  $\pi^*B'_2$ is also $\sigma_7^*$-invariant and
  $$\pi^*B'_2=\sigma_3^*(\pi^*B'_1)\equiv\sigma_3^*(2L)+\sigma_3^*(t)=2L+\sigma_3^*(t),$$
we see that $\sigma_3^*(t)$ is also $\sigma_7^*$-invariant.  We know that
$$\pi_1(Y)\cong\pi_1(\tilde{Y})\cong C_2,$$ in particular, $X$ may have at most one $\sigma$-invariant torsion, and in that case it must be a 2-torsion. Since
both  $t$ and $\sigma_3^*(t)$ are $\sigma_7^*$-invariant, either $t$ is trivial or 
a $\sigma_3^*$-invariant 2-torsion. It follows that
$$\pi^*B'_1\equiv \pi^*B'_2\equiv\pi^*B'_3\equiv 2L+t.$$
Furthermore, the intersection numbers $k_{ij}:=A_{i3}B_j=\delta_{ij}$. This follows from the exactly same computation as in Theorem \ref{2,3}, since  $A_{i3}$ is a 4-section and  
$2=A_{i3}F_0=k_{i1}+k_{i2}+k_{i3}+1$. The rest is ditto, except that
we choose a section $g_j\in H^0(X, 2L+t)$ that gives the divisor $\pi^*B'_j$. Since $2t$ is trivial, the 4 sections
$$g_1^2,\, g_2^2,\, g_1g_2,\, g_3^2\in H^0(X, 4L).$$  
\end{proof}

\bigskip
{\bf Proof of Theorem \ref{main}.}

\medskip
In this case the proof is simpler than the previous case. If there is a $\sigma_7^*$-invariant curve $C\in |2L|,$ then the same argument shows that the corresponding curve $\tilde{C}\subset \tilde{Y}$ is contained in a union of fibres. On the other hand, the $I_9$-fibre has multiplicity 1 and hence 
$$B'_1+B'_2+B'_3\sim 4K_{Y},$$
$$\pi^*B'_1+\pi^*B'_2+\pi^*B'_3\sim 4\pi^*K_{Y}\equiv 4K_X\sim 12L$$
hence
$$\pi^*B'_1\sim\pi^*B'_2\sim\pi^*B'_3\sim 4L.$$
If $F_1$ is the reduced curve of the fibre with multiplicity 2 or 4, then it is irreducible and the image
$F_1'\subset Y$ satisfies
$F_1'\sim 2K_{Y}$ or $K_Y$ and hence
$$\pi^*F_1'\sim 6L\,\,\,{\rm or}\,\,\,3L.$$
Thus no union of irreducible components of fibres can be equal to 
the curve $\tilde{C}$. The proof is completed.

\begin{remark}\label{rem24} In \cite{K11}, Remark 3.3, the intersection numbers $$k_{ij}:=A_{i3}B_j$$
were computed, where $A_{11}, A_{12}, B_1, A_{21}, A_{22}, B_2, A_{31}, A_{32},
B_3$ are the nine components of the $I_9$-fibre $F_0$ as before.  Using the same notation as before, we
get
$$k_{i1}+k_{i2}+k_{i3}=3,$$ $$m_j=k_{1j}+k_{2j}+k_{3j}+1,$$
$$m_j^2+14=3(k_{1j}^2+k_{2j}^2+k_{3j}^2)+11+4k_{jj}+2k_{j+1,j}.$$
There are many solutions to this system. Among them two are symmetric with respect to the order 3 rotation $(A_{13},
B_1)\to (A_{23}, B_2)\to (A_{33}, B_3)\to (A_{13}, B_1)$;
$$(k_{ij})=
\begin{pmatrix} 1&2&0\\0&1&2 \\ 2&0&1 \end{pmatrix},\,\,\,\begin{pmatrix}0&1&2\\2&0&1 \\1&2&0 \end{pmatrix}.$$ Each of the two cases, if exists, leads to a fake projective plane (\cite{K11} Section 5). According to the computer based computation by Cartwright and Steger \cite{CS2} there is exactly one pair of fake projective planes, complex conjugate to each other, with $\Aut(X)=7:3$ such that the order 7 quotient has fundamental group of order 2. Thus at least one of the two cases for $(k_{ij})$ occurs. If both occur, then  the corresponding fake projective planes must be complex conjugate to each other.
\end{remark}

\subsection{The case of $(3,3)$-elliptic surface}

Cartwright and Steger have computed the fundamental groups of all quotients of fake projective planes \cite{CS2}. According to their group theoretic computation based on computer, the quotient of a fake projective plane by an order 7 automorphism, if it admits, has fundamental group either trivial or cyclic of order 2. This eliminates the possibility of $(3,3)$-elliptic surface case. However we will show that our argument works  even in this case, instead of referring to their work.

First we note that all statements of Theorem \ref{2,4} hold, even if $(2,4)$ is replaced by $(3,3)$. Indeed, for a $(3,3)$-elliptic surface the general fibre $F$ is numerically equivalent to $3K$, thus any $(-3)$-curve is a 3-section, since it has intersection number 1 with $K$.  Then by Proposition \ref{3} the $I_9$-fibre on $\tilde{Y}$ cannot have multiplicity 3, hence 1. 

If there is a $\sigma_7^*$-invariant curve $C\in |2L|$, then the same argument shows that the corresponding curve $\tilde{C}\subset\tilde{Y}$ is contained in a union of fibres. Since a general fibre is numerically equivalent to $3K_{\tilde{Y}}$, we see that
$$A_{11}+ A_{12}+ B_1+ A_{21}+ A_{22}+B_2+ A_{31}+A_{32}+B_3=F_0\sim 3K_{\tilde{Y}},$$
$$B'_1+B'_2+B'_3\sim 3K_{Y},$$
$$\pi^*B'_1+\pi^*B'_2+\pi^*B'_3\sim 3\pi^*K_{Y}\equiv 3K_X,$$
hence
$$\pi^*B'_1\sim\pi^*B'_2\sim\pi^*B'_3\sim K_X\sim 3L.$$
If $F_1$ is the reduced curve of the fibre with multiplicity 3, then it is irreducible and $F_1\sim K_{\tilde{Y}}$, so the image  $F_1'\subset Y$ satisfies $F_1'\sim K_{Y}$ and hence $$\pi^*F_1'\sim K_X\sim 3L.$$ Thus no union of irreducible components of fibres can be equal to the curve $\tilde{C}$.

\section{Proof of Theorem \ref{main2}}

Throughout this section $X$ will denote a fake projective plane with $\Aut(X)\cong C_3^2$. \\
According to the computation of Cartwright and Steger \cite{CS2}, $X$ has $$H_1(X, \mathbb{Z})=C_7,\,\,C_{14}\,\,\,{\rm or}\,\,C_2^{2}\times C_{13}.$$  In particular $X$ has no 3-torsion, hence has a unique cubic root of $K_X$. 

Let $L$ be an ample generator of $\Pic(X)$  such that $\tau^*(2L)\cong 2L$ for all $\tau\in 
Aut(X)$. Such an $L$ exists, e.g. a cubic root of $K_X$ (Lemma \ref{L0}). 

Pick two automorphisms $\sigma$ and $\sigma'$ such that
$$\Aut(X)=\langle \sigma,\,\,\, \sigma' \rangle\cong C_3^2.$$
By Lemma \ref{H02}, it is enough to show that $H^0(X, 2L)=0$. Suppose that $$H^0(X, 2L)\neq 0.$$  Let $x_1,\, x_2,\, x_3\in X$ be the three fixed points of $\sigma$, $${\rm Fix}(\sigma)=\{x_1,\,\, x_2,\,\, x_3\}\subset X.$$ Then $\sigma'$  rotates $x_1, x_2, x_3$ and fixes three points, different from $x_i$'s (see subsection (1.1) or \cite{K08}). Note that $\Aut(X)$  acts on the projective space $${\bf P}H^0(X, 2L).$$ An automorphism of finite order of a projective space has a fixed point, thus there is a curve, possibly reducible, $$C\in |2L|\,\,\,{\rm with}\,\,\,\sigma^*(C)=C.$$ Let $$\pi: X\to X/\langle\sigma\rangle=Y$$ be the quotient map. Then there is a curve $$C'\subset Y=X/\langle\sigma\rangle$$ such that $C=\pi^*C'$. Since $C'^2=4/3$, we see that
$$C'\sim \dfrac{2}{3}K_Y.$$ Let $y_i\in Y$ be the image of $x_i$. Then $y_i$ is a singular point of type $\frac{1}{3}(1,2)$. Note that
$$C\equiv \sigma'^*(C)\equiv \sigma'^{2*}(C)\equiv 2L.$$

\medskip
{\bf Claim}: $C$ passes through exactly two of the three points $x_1, x_2, x_3\in X$. 

\medskip\noindent
This is equivalent to say that $C'$ passes through exactly two of the three singular points $y_1, y_2, y_3\in Y$. To see this, 
let $A_{i1}, A_{i2}\subset \tilde{Y}$ be the exceptional $(-2)$-curves over $y_i$.
Let $\tilde{C}\subset \tilde{Y}$ be the proper transform of $C'$ under the resolution $$f: \tilde{Y}\to Y.$$ Then
$$\tilde{C}\sim f^*C'-\sum_{k=1}^3{\mathcal{C}_k}$$
where $\mathcal{C}_k$ is a $\mathbb{ Q}$-divisor supported on  $A_{k1}\cup A_{k2}$, which can be computed as follows:
$$A_{k1}\tilde{C}=a,\,\, A_{k2}\tilde{C}=b \Longleftrightarrow \mathcal{C}_k=\frac{(2a+b)}{3}A_{k1}+\frac{(a+2b)}{3}A_{k2}.$$ In this case 
$$\mathcal{C}_k^2=-\frac{2}{3}(a^2+b^2+ab)=-\frac{2}{3}, \,\,-\frac{6}{3},\,\,-\frac{8}{3},\,\,-\frac{14}{3},\ldots.$$ It follows that 
$$\mathcal{C}_k^2=-\frac{2}{3}\Longleftrightarrow a+b=A_{k1}\tilde{C}+A_{k2}\tilde{C}=1 \Longleftrightarrow {\rm mult}_{x_k}C=1.$$
We compute
$$\tilde{C}^2=C'^2+\sum_{k=1}^3{\mathcal{C}_k^2}=\frac{4}{3}+\sum_{k=1}^3{\mathcal{C}_k^2}.$$ Note that $C$ passes through $x_k$ if and only if $\mathcal{C}_k\neq 0$. It is easy to check that $\frac{4}{3}-\frac{2}{3}(a^2+b^2+ab)$ is not an integer for any integer $a, b\ge 0$.   Since $\tilde{C}^2$ is an integer, at least two of $\mathcal{C}_1$,  $\mathcal{C}_2$,  $\mathcal{C}_3$ are non-zero, i.e. $C$ passes through at least two of the three points $x_1, x_2, x_3$. Suppose $x_1, x_2, x_3\in C$. 
If $C$ has multiplicity $\ge 2$ at $x_3$, then the curve $\sigma'^*(C)$ has multiplicity $\ge 2$ at $x_2$, hence
$$4=\sigma'^*(C)C\ge\sum_{k=1}^3{\rm mult}_{x_k}\sigma'^*(C)\cdot {\rm mult}_{x_k}C\ge 1+2+2,$$ absurd. If $C$ has multiplicity 1 at each $x_k$, then $\tilde{C}^2=\frac{4}{3}-\frac{2}{3}\cdot 3=-\frac{2}{3}$, not an integer. 
This proves Claim. 

Now we may assume that $x_1,\, x_2 \in C,\,\,x_3\notin C.$ Then
$$x_3,\, x_1 \in  \sigma'^*(C),\,\,x_2\notin  \sigma'^*(C),$$ $$x_2,\, x_3 \in  \sigma'^{2*}(C),\,\,x_1\notin  \sigma'^{2*}(C).$$  This proves that $\dim H^0(X, 2L)\ge 3$, hence $h^0(X, 4L)\ge 4$, contradicting Lemma \ref{4L}. Indeed,   
if $g_1, g_2, g_3\in H^0(X, 2L)$ are sections giving the divisor $C$, $\sigma'^*(C)$, $\sigma'^{2*}(C)$, respectively, then the 4 sections
$$g_1^2,\, g_2^2,\, g_1g_2,\, g_3^2\in H^0(X, 4L)$$  
are linearly independent. 

\section{Exceptional sequences on a fake projective plane}

Let 
$D^b(coh(W))$ denote the bounded derived category of coherent sheaves on a smooth variety $W$. It is a triangulated category.

An object $E$ in a triangulated category is called exceptional if $Hom(E, E[i])=\mathbb{C}$ if $i=0$, and $=0$ otherwise.

A sequence $E_1,\ldots, E_n$ of exceptional objects is called an exceptional sequence if  $Hom(E_j, E_k[i])=0$ for any $j>k$, any $i$.

When $W$ is a smooth surface with $p_g=q=0$, every line bundle is an exceptional object in $D^b(coh(W))$.

Let $X$ be a fake projective plane and $L$ be an ample generator of $\Pic(X)$. The three line bundles
$$2L,\,\,L,\,\,\mathcal{O}_X$$ form an exceptional sequence if and only if $H^j(X, 2L)=H^j(X, L)=0$ for all $j$. Thus Corollary \ref{cor} follows from Theorem 
\ref{main} and \ref{main2}.

 Write
$$D^b(coh(X))=
\langle 2L, \,\,L,\,\, \mathcal{O}_X,\,\,\mathcal{A}\rangle$$
where $\mathcal{A}$ is the orthogonal complement of the admissible triangulated subcategory generated by $2L, \,\,L,\,\, \mathcal{O}_X$. Then the Hochshield homology $$HH_*(\mathcal{A})=0.$$ This can be read off from the Hodge numbers. In fact, the Hochshield homology of $X$ is the direct sum of Hodge spaces $H^{p,q}(X)$, and its total dimension is the sum of all Hodge numbers. The latter is equal to the topological Euler number $c_2(X)$, as a fake projective plane has Betti numbers $b_1(X)=b_3(X)= 0$.  

The Grothendieck group $K_0(X)$ has filtration
$$K_0(X)=F^0K_0(X)\supset F^1K_0(X)\supset F^2K_0(X)$$ with
$$F^0K_0(X)/F^1K_0(X)\cong CH^0(X)\cong \mathbb{Z},$$
$$F^1K_0(X)/F^2K_0(X)\cong \Pic(X),$$
$$F^2K_0(X)\cong CH^2(X).$$
If the Bloch conjecture holds for $X$, i.e. if  $CH^2(X)\cong \mathbb{Z}$, then 
 $K_0(\mathcal{A})$ is finite.

%%%%%%%%%%%%%% reference %%%%%%%%%%%%%%%%%

\end{document}